\documentclass[a4paper,12pt]{amsart}
\usepackage[lmargin=2.4cm,rmargin=2.25cm,bmargin=3cm,tmargin=2.5cm]{geometry}

\usepackage[utf8]{inputenc}
\usepackage{amsmath}
\usepackage{mathtools}

\usepackage{amssymb}
\usepackage{amsthm}
\usepackage{textcomp}
\usepackage{enumerate}
\usepackage[numbers]{natbib}
\usepackage{array}
\usepackage{verbatim}
\usepackage{hyperref}
\usepackage{graphicx}
\usepackage{booktabs}
\usepackage{algorithm}
\usepackage{algpseudocode}
\usepackage{xcolor}

\usepackage{tikz}
\usetikzlibrary{automata, positioning}

\hypersetup{
  colorlinks,
  linkcolor={red!50!black},
  citecolor={blue!50!black},
  urlcolor={blue!80!black}
}

\usepackage{tcolorbox}
\usepackage{soul}
\usepackage[foot]{amsaddr}

\usepackage{todonotes}

\theoremstyle{plain}
\newtheorem{theorem}{Theorem}
\newtheorem{corollary}[theorem]{Corollary}
\newtheorem{proposition}[theorem]{Proposition}
\newtheorem{lemma}[theorem]{Lemma}

\newtheorem{assumption}{Assumption}

\newtheorem{example}[theorem]{Example}

\theoremstyle{remark}
\newtheorem{remark}[theorem]{Remark}

\newcommand{\ud}{\mathrm{d}} 
\newcommand{\R}{\mathbb{R}}  
\newcommand{\N}{\mathbb{N}}  
\renewcommand{\P}{\mathbb{P}} 
\newcommand{\E}{\mathbb{E}}  
\newcommand{\uarg}{\;\cdot\;} 
\newcommand{\charfun}[1]{\mathbf{1}\left(#1\right)} 

\newcommand{\X}{\mathbb{X}}  
\newcommand{\Sp}{\mathbb{S}}  
\newcommand{\sigmaX}{\mathcal{X}}
\newcommand{\sigmaSp}{\mathcal{S}}
\newcommand{\F}{\mathcal{F}}  
\newcommand{\given}{\;:\;}
\newcommand{\osc}{\mathrm{osc}}
\newcommand{\tv}{\mathrm{tv}}

\newcommand{\aseq}{\overset{\mathrm{a.s}}{=}}

\newcommand{\hlfirst}[1]{{\color{blue}#1}}
\newcommand{\hlsecond}[1]{{\color{red}#1}}

\definecolor{lightblue}{rgb}{0.7,0.7,1}

\newcommand{\braces}[1]{\lbrace #1 \rbrace}

\newcommand{\biggbraces}[1]{\bigg\lbrace #1 \bigg\rbrace}

\title[Invitation to adaptive MCMC convergence theory]{An invitation to adaptive Markov chain Monte Carlo convergence theory}

\author{Pietari Laitinen$^1$}
\author{Matti Vihola$^1$}
\address{$^1$ Department of Mathematics and Statistics, University of Jyväskylä, Finland}

\begin{document}

\begin{abstract}
   Adaptive Markov chain Monte Carlo (MCMC) algorithms, which automatically tune their parameters based on past samples, have proved extremely useful in practice. The self-tuning mechanism makes them `non-Markovian', which means that their validity cannot be ensured by standard Markov chains theory. Several different techniques have been suggested to analyse their theoretical properties, many of which are technically involved. The technical nature of the theory may make the methods unnecessarily unappealing. We discuss one technique---based on a martingale decomposition---with uniformly ergodic Markov transitions. We provide an accessible and self-contained treatment in this setting, and give detailed proofs of the results discussed in the paper, which only require basic understanding of martingale theory and general state space Markov chain concepts. We illustrate how our conditions can accomodate different types of adaptation schemes, and can give useful insight to the requirements which ensure their validity.
\end{abstract}



\maketitle

\section{Introduction}

The adaptive Metropolis (AM) algorithm \citep{saksman-am} was the first instance of a class of algorithms which are known as adaptive Markov chain Monte Carlo (MCMC). The AM algorithm calculates an estimate of the target distribution's covariance matrix recursively, and uses it within the random-walk proposal distribution. The AM algorithm is, in many ways, prototypical adaptive MCMC: past simulated samples are used `as samples from the target' to progressively learn (and optimise) an easily estimatable proxy, which is used to improve the sampler's mixing. In the case of the AM, the celebrated optimal scaling result \citep{roberts-gelman-gilks-scaling} suggests that proposal covariance proportional to the target covariance should imply good mixing; recent analysis \citep{andrieu-lee-power-wang} consolidates this in terms of spectral gaps.

Since the AM algorithm, a number of adaptive MCMC algorithms have been introduced. For instance, random-walk Metropolis (RWM) proposal scale can be adapted based on the acceptance rate \citep{andrieu-robert} and the scale adaptation can be combined with covariance learning \citep{andrieu-thoms}. Alternatively, the target covariance can be learnt by a directional acceptance rate rule \citep{vihola-ram}. The RWM adaptation has been used together with delayed rejection \citep{haario-laine-mira-saksman} and component/block-wise sampling \citep{saksman-jrstatsoc}. Adaptation has been applied in other MCMC updates, for instance including automatic temperature selection in parallel tempering \citep{miasojedow-moulines-vihola}, scanning probabilities of random scan Gibbs samplers \citep{latuszynski-roberts-rosenthal}, with Langevin-type proposals \citep{marshall-roberts}, within a conditional particle filter's initialisation \citep{karppinen-vihola}, choosing tolerances in approximate Bayesian computing (ABC) MCMC \citep{vihola-franks-abc}, with reversible jump transitions \citep{tian-lee-zhou} and using reinforcement learning to learn proposals \citep{wang-chen-kanagawa-oates}. Adaptation has been suggested also with parallel Markov chains and to use different strategies in different regions of space \citep{craiu-rosenthal-yang}, specifically for multimodal distributions \citep{pompe-holmes-latuszynski} and for Bayesian variable selection \citep{griffin-latuszynski-steel}.

In the original paper \cite{saksman-am}, the validity of the AM algorithm (a strong law of large numbers) was established using mixingales, a from of `approximate martingales'. The mixingale approach was extended and abstracted further in \cite{atchade-rosenthal}. Martingale decomposition (or Poisson equation) approach was introduced in \cite{andrieu-moulines} for the analysis of adaptive MCMC, which use `stochastic gradient type' learning rules. Coupling technique was formalised in \cite{roberts-rosenthal} to analyse adaptive MCMC (convergence in law and weak law of large numbers). Another martingale approach, based on resolvents, was used in \cite{atchade-fort}. Recently, \cite{chimisov-latuszynski-roberts} suggested to replace continuous adaptation by increasingly rare updates of the parameter, and analysed the algorithms with regeneration arguments; see also the recent work \cite{hofstadler-latuszynski-roberts-rudolf} who investigate such algorithms using a similar martingale decomposition as the present paper.

We present a simple but broadly applicable framework for adaptive MCMC (Section \ref{sec:framework}), and write a decomposition of the ergodic averages into three terms, following the approach in \citep{andrieu-moulines}. We then discuss how two of the terms in this decomposition can be controlled under a so-called \emph{simultaneous uniform ergodicity} condition (Section \ref{sec:uniform}). The remaining term, which captures the `perturbation' caused by changing Markov transitions due to adaptation, can be controlled under additional assumptions, which reflect different types of adaptation mechanisms. The abstract assumption we adopt is \emph{waning adaptation}, which leads to weak/strong laws of large numbers and central limit theorems (Section \ref{sec:waning}). 

We then turn to different types of adaptation dynamics and conditions related to them. A common recipe to construct adaptive MCMC algorithms is based on stochastic gradient type (or Robbins--Monro stochastic approximation) updates (Section \ref{sec:sa}), which can often be shown to satisfy waning adaptation---assuming a suitable continuity of the Markov updates. An alternative strategy, which avoids such continuity constraints, is based on increasingly rare adaptation introduced in \citep{chimisov-latuszynski-roberts}, which also implies waning adaptation (Section \ref{sec:air}). We then discuss how \emph{independent} adaptation, leading effectively into inhomogeneous Markov chains, fit our framework (Section \ref{sec:inhomog}). We briefly discuss some extensions and generalisations of the results presented in the paper (Section \ref{sec:extensions}) and conclude with a discussion (Section \ref{sec:discussion}).

\section{Generic adaptive MCMC and a martingale decomposition}
\label{sec:framework}

Throughout the paper, we assume that $\pi$ is a probability distribution on a general state space $\X$, and $\{P_s\}_{s\in\Sp}$ is a family of Markov transition probabilities on $\X$ indexed by a parameter $s\in\Sp$, which will be chosen adaptively. For each $s\in\Sp$, the transition $P_s$ is assumed to be $\pi$-invariant.

\begin{remark}
   \label{rem:measurability}
To be precise, we assume that $\Sp$ and $\X$ are measurable spaces, equipped with $\sigma$-algebras $\sigmaSp$ and $\sigmaX$, respectively, and product spaces are equipped with the corresponding product $\sigma$-algebras. The family $\{P_s\}_{s\in\Sp}$ must satisfy the following (non-restrictive) regularity condition: $(s,x) \mapsto P_s(x,A)$ is $\sigmaSp\otimes\sigmaX$-measurable for all $A\in\sigmaX$, which ensures that $\big((s,x), A\big) \mapsto P_s(x,A)$ defines a probability kernel from $\Sp\times\X$ to $\X$.
\end{remark}

The abstract adaptive MCMC is a $\Sp\times\X$-valued stochastic process $(S_k,X_k)_{k\ge 0}$, which satisfies the following assumption:
\begin{assumption}
   \label{a:markov}
   $(\F_k)_{k\ge 0}$ are a filtration such that $(S_k,X_k)_{k\ge 0}$ is $(\F_k)_{k\ge 0}$-adapted\footnote{that is, $\F_k$ are increasingly refinbed $\sigma$-algebras $\F_k\subset\F_{k+1}$ such that $(S_k,X_k)$ are $\F_k$-measurable.}, and the following holds for all $k\ge 0$ and all measurable $A\subset\X$:
   $$
   \P(X_{k+1}\in A\mid \F_k) = P_{S_k}(X_k,A), \qquad \text{a.s.}
   $$
\end{assumption}

Assumption \ref{a:markov} abstracts and generalises iterative algorithms with updates of the following form \citep[cf.][]{roberts-rosenthal}:
\begin{itemize}
   \item sample $X_{k+1} \sim P_{S_k}(X_k, \uarg)$, and
   \item update $S_{k+1}$ based on $X_0,\ldots,X_{k+1}$ and $S_k$.
\end{itemize}
If the algorithm is exactly of this type, we may take $(\F_k)$ as the natural filtration of $(S_k,X_k)$ in \ref{a:markov}. However, in many adaptive MCMC algorithms, sampling might involve a number of auxiliary variables which are used in the adaptation step. The filtration $\F_k$ can accomodate the information about such auxiliary random variables, without imposing any unnecessary constraints or complications to the definition of the Markov transitions $P_s$.

We adopt the common notation where $\pi(f) = \int f(x) \pi(\ud x)$ and $(Pf)(x) = \int P(x,\ud y) f(y)$, whenever well-defined, and
$L^1(\pi)$ stands for the set of all measurable functions $f:\X\to\R$ which are $\pi$-integrable.  
In what follows, we assume that $\varphi\in L^1(\pi)$ is a fixed test function, and investigate the convergence properties of the averages
\begin{equation}
\frac{1}{n}\sum_{k=1}^n \varphi(X_k) \xrightarrow{n\to\infty} \pi(\varphi) ,
\label{eq:lln}
\end{equation}
either in probability (weak law of large numbers) or almost surely (strong law of large numbers).

The following is an abstract assumption, which is central to our analysis:
\begin{assumption}[Solutions of the Poisson equation]
   \label{a:poisson}
There exists a measurable mapping $(s,x)\mapsto g_s(x)$ from $\Sp\times\X$ to $\R$ which satisfies the following:
$$
   g_s(x) - (P_s g_s)(x) = \varphi(x) - \pi(\varphi) \qquad \text{for all $x\in\X$ and $s\in\Sp$}.
$$
\end{assumption}
It turns out that such a function exists under general conditions (see Theorem \ref{thm:bounded} in Section \ref{sec:uniform}).
Having \ref{a:poisson} allows us to investigate \eqref{eq:lln} in terms of the following decomposition:
\begin{align}
   \sum_{k=1}^n \big[\varphi(X_k) - \pi(\varphi)\big] 
   &= 
   \sum_{k=1}^n \big[g_{S_k}(X_k) - P_{S_k}g_{S_k}(X_k)\big] \label{eq:main-decomposition} \\
   &= 
   \sum_{k=1}^n \big[\hlfirst{g_{S_{k-1}}(X_k)} - \hlsecond{P_{S_{k-1}}g_{S_{k-1}}(X_{k-1})}\big] 
   & \big(=M_n\big) 
   \nonumber\\
   & + 
   \sum_{k=1}^n \big[g_{S_k}(X_k) - \hlfirst{g_{S_{k-1}}(X_k)}\big] 
   & \big(=A_n\big) 
   \nonumber\\
   & + 
   \sum_{k=1}^n \big[\hlsecond{P_{S_{k-1}}g_{S_{k-1}}(X_{k-1})} - P_{S_k}g_{S_k}(X_k)\big] &\big(=R_n\big),
   \nonumber
\end{align}
where the colors indicate the matching terms which have been added and subtracted. With this decomposition, we may write
$$
\frac{1}{n}\sum_{k=1}^n \varphi(X_k) - \pi(\varphi) = \frac{M_n}{n} + \frac{A_n}{n} + \frac{R_n}{n},
$$
and we detail later how each of these terms can be controlled under additional assumptions.

Before that, let us discuss their intuitive role and typical behaviour. The last `remainder` term $R_n$ is a telescoping sum, which equals
\begin{equation}
   R_n =  P_{S_0}g_{S_0}(X_0) - P_{S_n} g_{S_n}(X_n),
   \label{eq:telescopic}
\end{equation}
and therefore, as long as $P_{S_k}g_{S_k}(X_k)$ are stable in $k$, the contribution of $R_n/n$ will be negligible for large $n$.

The second `adaptation' term $A_n$ contains the `disturbances' caused by the adaptation, because of the changing Markov transitions. If the adaptation slows down, typically $A_n/n$ vanishes as $n$ increases. Note also that in the case of a standard (non-adaptive) MCMC, where the terms $S_k = S_0\equiv s$ for all $k$, we have $A_n=0$. 

Finally, it is not hard to verify that (whenever well-defined and suitably regular) the conditional expectation of the $k$th term of $M_n$ with respect to $\F_{k-1}$ is zero, and therefore $M_n$ is a zero-mean martingale. It captures the bulk of the random fluctuations, and $M_n/n\to 0$ often follows easily from a martingale convergence theorem.

In fact, the martingale terms $M_n$ can satisfy a central limit theorem $M_n/\sqrt{n} \to N(0,\sigma_\varphi^2)$ in distribution for some $\sigma_\varphi^2\in[0,\infty)$. Therefore, if $R_n/\sqrt{n}$ and $A_n/\sqrt{n}$ vanish, we obtain a central limit theorem:
\begin{equation}
   \frac{1}{\sqrt{n}}\sum_{k=1}^n \big[\varphi(X_k) - \pi(\varphi)\big] = \frac{M_n}{\sqrt{n}} + \frac{A_n}{\sqrt{n}} + \frac{R_n}{\sqrt{n}} \xrightarrow{n\to\infty} N(0,\sigma_\varphi^2).
   \label{eq:clt}
\end{equation}

\section{Simultaneous uniform ergodicity}
\label{sec:uniform}

For the sake of exposition, we focus here on the case where the Markov transitions $P_s$ are uniformly ergodic. Many practical MCMC algorithms, such as the random-walk Metropolis, are uniformly ergodic when used in a compact state space, but some algorithms can be uniformly ergodic on non-compact spaces, too \citep[e.g.][]{andrieu-lee-vihola}. We discuss later how some of the results generalise to other forms of ergodicity (Section \ref{sec:non-uniform}).

In what follows, $L^\infty$ stands for the bounded measurable functions $f:\X\to\R$, and we denote $\| f \|_\infty = \sup_x |f(x)|$ and $\osc(f) = \sup_x f(x) - \inf_x f(x)$. Bounded functions that have zero mean under $\pi$ are denoted $L_0^\infty(\pi) = \{f\in L^\infty\given \pi(f)=0\}$ and the $\pi$-centred version of $f\in L^\infty$ is denoted by $\bar{f}(x) = f(x) - \pi(f) \in L_0^\infty(\pi)$. The total variation distance for two probability measures $\mu,\nu$ on $\X$ is defined as 
$$
d_\tv(\mu,\nu) 
= \sup_A | \mu(A)-\nu(A)| 
= \frac{1}{2}\sup_{\|f\|_\infty \le 1} |\mu(f)-\nu(f)|
= \sup_{\mathrm{osc}(f)\le 1} |\mu(f) - \nu(f)|.
$$

The following condition \cite[cf.][]{saksman-am,roberts-rosenthal} ensures that the Markov transitions $P_s$ are uniformly ergodic with the same constants.
\begin{assumption}[Simultaneous uniform ergodicity]
   \label{a:simultaneous-uniform}
   There exist constants $C<\infty$ and $\rho\in[0,1)$ such that for all $x\in\X$, $s\in\Sp$ and $k\ge 0$:
$$
d_\tv(P_s^k(x, \uarg), \pi)
\le C \rho^k.
$$
\end{assumption}
Theorem \ref{thm:bounded} below shows that \ref{a:simultaneous-uniform} implies \ref{a:poisson}, which we precede by the following well-known result about how the mappings $P_s^k:L_0^\infty(\pi)\to L_0^\infty(\pi)$ behave under \ref{a:simultaneous-uniform}.  

\begin{lemma}\label{lemma:bounded}
Suppose \ref{a:simultaneous-uniform} holds. Let $f\in L_0^\infty(\pi)$, $k\in \N$ and $s\in \Sp$, then $P_s^kf\in L_0^\infty(\pi)$ and
   $$
     \| P_s^k f \|_\infty \le C \rho^k \osc ( f) \le 2C \rho^k \| f\|_\infty. 
   $$
\end{lemma}
\begin{proof}
The result holds trivially for $f\equiv 0$. Otherwise, let $h=f/\osc(f)$, for which $\osc(h) = 1$, and we may write
\begin{align*}
\sup_{x\in \X}|P_s^k f(x)|
&= \osc(f) \sup_{x\in \X} |P_s^kh(x)-\pi(h)| \\
&\le \osc(f) \sup_{x\in \X}\sup_{\osc(g) \le 1} |P_s^kg(x)-\pi(g)| \\
&=\osc(f) \sup_{x\in \X}d_\tv\big(P_s^k(x,\uarg), \pi\big),
\end{align*}
from which the first inequality follows. The latter follows because $\pi(f)=0$ implies $\inf_x f(x)\le 0 \le \sup_x f(x)$. We have $\pi(P_s^k f)= \pi(f)=0$, since $P_s$ is $\pi$-invariant, and therefore $P_s^kf\in L_0^\infty(\pi)$ by Lemma \ref{lem:gs-measurability} in Appendix \ref{app:measurability}. 
\end{proof}

\begin{theorem}\label{thm:bounded}
   Let $\varphi\in L^\infty$ and assume \ref{a:simultaneous-uniform}. Then, \ref{a:poisson} holds with
   \begin{equation}
   g_s(x)
   = \sum_{k=0}^\infty P_s^k \bar{\varphi}(x).
   \label{eq:poisson-neumann}
   \end{equation}
   Furthermore, for all $s\in\Sp$ it holds that $g_s\in L_0^\infty(\pi)$ and
   $$
      \|g_s\|_\infty \le \frac{C}{1-\rho} \osc(\varphi).
   $$
\end{theorem}
\begin{proof}
Each term of the sum is measurable by Lemma \ref{lem:gs-measurability} in Appendix \ref{app:measurability}. By Lemma \ref{lemma:bounded}
\begin{equation*}
\sum_{k=0}^\infty \|P_s^k\bar{\varphi}\|_\infty
\le \sum_{k=0}^\infty \osc(\varphi) C\rho^k=
\osc(\varphi) \frac{C}{1-\rho},
\end{equation*}
which ensures that \eqref{eq:poisson-neumann} converges and is measurable and
\begin{equation*}
\|g_s\|_\infty
=\bigg\|\sum_{k=0}^\infty P_s^k\bar{\varphi}\bigg\|_\infty\le \osc(\varphi) \frac{C}{1-\rho}.
\end{equation*}
Thus by the dominated convergence theorem
\begin{equation*}
\int g_s(x) \pi(\ud x)= \sum_{k=0}^\infty \int P_s^k \bar{\varphi}(x) \pi(\ud x)=0.
\qedhere
\end{equation*}
\end{proof}

For the rest of this section, we assume that both \ref{a:markov} and \ref{a:simultaneous-uniform} hold, and $g_s$ in \ref{a:poisson} are as in Theorem \ref{thm:bounded}.

\begin{lemma}
   \label{lemma:martingale}
Let $g_s$ be defined as in \eqref{eq:poisson-neumann}
then for all $k\ge 1$, the terms
$$
   \Delta_k = g_{S_{k-1}}(X_k) - P_{S_{k-1}}g_{S_{k-1}}(X_{k-1})
$$
are bounded, $\F_k$-measurable and satisfy:
\begin{align*}
      \E[\Delta_k \mid \mathcal{F}_{k-1}] & \aseq  0 \\
      \E[\Delta_k^2 \mid \mathcal{F}_{k-1}] & \aseq 
      P_{S_{k-1}} g_{S_{k-1}}^2 (X_{k-1}) - (P_{S_{k-1}} g_{S_{k-1}})^2(X_{k-1}).
\end{align*}
\end{lemma}
\begin{proof}
The boundedness and measurability are direct and 
\begin{align*}
\E[\Delta_k \mid \mathcal{F}_{k-1}]
&=\E[g_{S_{k-1}}(X_k) \mid \mathcal{F}_{k-1}]-P_{S_{k-1}}g_{S_{k-1}}(X_{k-1}), 
\end{align*}
where $\E[g_{S_{k-1}}(X_k) \mid \mathcal{F}_{k-1}] = P_{S_{k-1}}g_{S_{k-1}}(X_{k-1})$ (a.s.)
by Lemma \ref{lem:h-measurable} in Appendix \ref{app:measurability}, which proves the first claim.

Similarly, we may calculate (a.s.)
\begin{align*}
&\E[\Delta_k^2 \mid \mathcal{F}_{k-1}] \\
&=\E\big[g_{S_{k-1}}^2(X_k) - 2g_{S_{k-1}}^2(X_k)P_{S_{k-1}}g_{S_{k-1}}(X_{k-1})+(P_{S_{k-1}}g_{S_{k-1}})^2(X_{k-1}) \mid \mathcal{F}_{k-1}\big] \\ 
&=\E[g_{S_{k-1}}^2(X_k) \mid \mathcal{F}_{k-1}]- 2P_{S_{k-1}}g_{S_{k-1}}(X_{k-1})\E[g_{S_{k-1}}(X_k) \mid \mathcal{F}_{k-1}]+(P_{S_{k-1}}g_{S_{k-1}})^2(X_{k-1}) \\ 
&=P_{S_{k-1}}g_{S_{k-1}}^2(X_k)- 2(P_{S_{k-1}}g_{S_{k-1}})^2(X_{k-1})+(P_{S_{k-1}}g_{S_{k-1}})^2(X_{k-1}) \\ 
&=P_{S_{k-1}}g_{S_{k-1}}^2(X_k)-(P_{S_{k-1}}g_{S_{k-1}})^2(X_{k-1}). \qedhere
\end{align*}
\end{proof}

Lemma \ref{lemma:martingale} establishes that $(\Delta_k)_{k\ge 1}$ are bounded martingale differences with respect to $(\F_k)_{k\ge 1}$, and therefore $M_n = \sum_{k=1}^n \Delta_k$ is a martingale. 

The following lemma ensures that the terms $R_n/n$ and $M_n/n$ vanish, as required by a law of large numbers. The term $R_n/\sqrt{n}$ vanishes, too, and $M_n/\sqrt{n}$ will converge to a Gaussian limit as shown in Lemma \ref{lem:martingale-term-clt}.

\begin{lemma}
   \label{lem:control-RM}
Let $g_s$ be defined as in \eqref{eq:poisson-neumann}, then the following hold as $n\to\infty$:
$$
\frac{R_n}{n} \to 0, \qquad \frac{R_n}{\sqrt{n}}\to 0 \qquad \text{and}\qquad \frac{M_n}{n} \to 0,\qquad\text{(a.s.)}
$$
\end{lemma}
\begin{proof}
Since $R_n$ is a telescopic \eqref{eq:telescopic}, we have for $p\in\{1,1/2\}$:
$$
\bigg|\frac{R_n}{n^p}\bigg|
=\frac{1}{n^p}\big|P_{S_{n}}g_{S_{n}}(X_{n})- P_{S_0}g_{S_0}(X_0)\big|
\le \frac{1}{n^p} \big( \| P_{S_{n}}g_{S_{n}} \|_\infty + \| P_{S_0}g_{S_0} \|_\infty \big) \to 0,
$$
because the functions are uniformly bounded (Theorem \ref{thm:bounded}). 

For the remaining term, consider the following martingale:
\begin{align*}
   V_n=\sum_{k=1}^n\frac{\Delta_k}{k},
\end{align*}
where the martingale differences $\|\Delta_k\|_\infty < M$ for some constant $M<\infty$ by Theorem \ref{thm:bounded}.

Because martingale differences are orthogonal, we have
$$
   \E[V_n^2]=\sum_{k=1}^n\frac{\E[\Delta_k^2]}{k^2} \le M^2 \sum_{k=1}^\infty \frac{1}{k^2} = M^2 \frac{\pi^2}{6}.
$$
That is, $V_n$ is a $L^2$-bounded martingale, which converges to an a.s.~finite $V_\infty = \sum_{k=1}^\infty\frac{\Delta_k}{k}$ almost surely (see Lemma \ref{lem:square-martingale-conv} in Appendix \ref{app:martingales}). Whenever $V_\infty$ is finite, Kronecker's lemma implies that 
\begin{equation*}
\frac{M_n}{n}=\frac{1}{n}\sum_{k=1}^n k \frac{\Delta_k}{k}\xrightarrow{n\to\infty} 0.
\qedhere
\end{equation*}
\end{proof}

We record the following abstract assumption and a lemma which ensures a central limit theorem for the martingale term.

\begin{assumption}
   \label{a:limiting-poisson}
Suppose that there exists a constant $\sigma_\varphi^2\in[0,\infty)$ such that
$$
   \frac{1}{n} \sum_{k=1}^n 
   \big[P_{S_{k-1}} g_{S_{k-1}}^2 (X_{k-1}) - (P_{S_{k-1}} g_{S_{k-1}})^2(X_{k-1})\big] 
   \xrightarrow{n\to\infty} \sigma_\varphi^2 \qquad \text{in probability.}
$$
\end{assumption}

\begin{lemma}
   \label{lem:martingale-term-clt}
Assume that \ref{a:limiting-poisson} holds, then
$$
   \frac{M_n}{\sqrt{n}} \xrightarrow{n\to\infty} N(0,\sigma_\varphi^2) \qquad \text{in distribution.}
$$
\end{lemma}
\begin{proof}
   The result follows from a martingale central limit theorem (Theorem \ref{thm:martingale-clt} in Appendix \ref{app:martingales}), because letting
   $\Delta_k = M_k - M_{k-1}$, by Lemma \ref{lemma:martingale}:
   $$
   \frac{1}{n} \sum_{k=1}^n \E[\Delta_k^2 \mid \F_{k-1}] 
   = \frac{1}{n} \sum_{k=1}^n 
   \big[P_{S_{k-1}} g_{S_{k-1}}^2 (X_{k-1}) - (P_{S_{k-1}} g_{S_{k-1}})^2(X_{k-1})\big] 
   \to \sigma^2,
   $$
   due to \ref{a:limiting-poisson}. By Theorem \ref{thm:bounded}
   $$
   |\Delta_k|\le \| g_{S_k} \|_\infty + \| P_{S_{k-1}}g_{S_{k-1}} \|_\infty 
   \le \| g_{S_k} \|_\infty + \| g_{S_{k-1}} \|_\infty \le 2 C (1-\rho)^{-1} \osc(\varphi),
   $$
   and therefore $\Delta_k^2 1(|\Delta_k| \ge \epsilon \sqrt{n})$ are identically zero for sufficiently large $n$.
\end{proof}

We present later verifiable conditions which imply \ref{a:limiting-poisson} (Lemma \ref{lem:convergence-continuity-clt} in Section \ref{sec:sa}).

\section{Waning adaptation}
\label{sec:waning}

Lemma \ref{lem:control-RM} in Section \ref{sec:uniform} showed that under \ref{a:markov} and \ref{a:simultaneous-uniform}, the terms $R_n/n$, $R_n/\sqrt{n}$ and $M_n/n$ in the decomposition \eqref{eq:main-decomposition} vanish as $n\to\infty$, and $M_n/\sqrt{n}$ converges to a Gaussian limit under \ref{a:limiting-poisson} (Lemma \ref{lem:martingale-term-clt}). 

In order to establish a law of large numbers and a central limit theorem, we only need to establish that $A_n/n\to 0$ and $A_n/\sqrt{n}\to 0$, respectively. The following general assumption will guarantee this:

\begin{assumption}[Waning adaptation]
   \label{a:waning}
   Let $D_k$ be random numbers taking values in $[0,1]$ such that 
   $$
   \sup_{x\in\X} d_\tv \big( P_{S_k}(x,\uarg), P_{S_{k-1}}(x,\uarg)\big) \le D_k.
   $$
   If the following holds for $p>0$:
   $$
      \frac{1}{n^p} \sum_{k=1}^n D_k \to 0 \qquad \text{(in probability/almost surely)},
   $$
   then the adaptation is said to be \emph{$p$-waning} (weakly/strongly, respectively).
\end{assumption}

\begin{theorem}
   \label{thm:waning-convergence}
Suppose that adaptation is weakly/strongly $p$-waning. Then,
$$
   \frac{A_n}{n^p} \xrightarrow{n\to\infty} 0,
$$
in probability/almost surely, respectively.
\end{theorem}

The proof of Theorem \ref{thm:waning-convergence} is preceded by the following key technical result, which establishes an upper bound for the individual terms in the sum $A_n$.

\begin{lemma}\label{lemma:ineq-D}
   Under \ref{a:simultaneous-uniform} and with $\varphi\in L^\infty$, the functions $g_s$ defined in Theorem \ref{thm:bounded} satisfy 
   $$
      \| g_{s} - g_{s'} \|_\infty \le \frac{4 C^2}{(1-\rho)^2} \osc(\varphi) \sup_{x\in\X} d_\tv\big( P_s(x,\uarg), P_{s'}(x,\uarg) \big)
      \qquad\text{for all }s,s'\in\Sp.
   $$
\end{lemma}
\begin{proof}
Let $s,s'\in\Sp$ and $n\ge 1$, and write the telescopic sum
\begin{align*}
P_s^n\bar{\varphi}(x)-P_{s'}^n\bar{\varphi}(x)
&=\sum_{i=0}^{n-1} \big(P_s^{i+1}P_{s'}^{n-(i+1)}\bar{\varphi}(x)-P_s^{i}P_{s'}^{n-i}\bar{\varphi}(x)\big)  \\
&=\sum_{i=0}^{n-1} P_s^{i}(P_s-P_{s'})P_{s'}^{n-i-1}\bar{\varphi}(x).
\end{align*}
For an operator $G: (L_0^\infty(\pi),\|\cdot \|_\infty)\to (L_0^\infty(\pi),\|\cdot \|_\infty)$, we use the operator norm:
$$\|G\|_{L_0^{\infty}(\pi)}=\sup_{h\in L_{0}^\infty(\pi),\, \|h\|_\infty \le 1}\|Gh\|_\infty.$$
We have by Lemma \ref{lemma:bounded}, that $(P_s-P_{s'})$ and $P_s^k$ for all $k\in \N$ are bounded operators. 
Therefore, by the sub-multiplicativity of the operator norm
\begin{align*}
\|P_s^n\bar{\varphi}-P_{s'}^n\bar{\varphi}\|_\infty 
&\le \sum_{i=0}^{n-1} \|P_s^{i}\|_{L_0^\infty(\pi)} \|P_s-P_{s'}\|_{L_0^\infty(\pi)}  \|P_{s'}^{n-i-1}\bar{\varphi}\|_\infty.
\end{align*}
The last term can be controlled by Lemma \ref{lemma:bounded}, which yields
\begin{align*}
\|P_{s'}^{n-i-1}\bar{\varphi}\|_\infty \le C \rho^{n-i-1} \osc(\varphi),
\end{align*}
and, similarly, the first term is upper bounded by
\begin{align*}
\|P_s^{i}\|_{L_0^\infty(\pi)}&=\sup_{h\in L_0^\infty(\pi),\,  \|h\|_\infty \le 1}\|P_s^{i}h\|_\infty 
\le 2 C \rho^i.
\end{align*}
For the remaining term, we may write 
\begin{align*}
\|P_s-P_{s'}\|_{L_0^\infty(\pi)}&=\sup_{h\in L_0^\infty(\pi),\,\|h\|_\infty \le 1}\|(P_s-P_{s'})h\|_\infty \\
&\le \sup_{\|h\|_\infty \le 1}\|(P_s-P_{s'})h\|_\infty \\
&=\sup_{x\in \X}  \sup_{\|h\|_\infty \le 1}|P_sh(x)-P_{s'}h(x)| \\
&\le 2\sup_{x\in \X} d_\tv\big(P_s(x,\uarg),P_{s'}(x,\uarg)\big).
\end{align*}
Let $D = \sup_{x\in\X} d_\tv\big(P_s(x,\uarg),P_{s'}(x,\uarg)\big)$, then combining these bounds yields
\begin{align*}
\|P_{s}^n\bar{\varphi}-P_{s'}^n\bar{\varphi}\|_\infty
& \le \sum_{i=0}^{n-1} 4C^2\rho^{n-1} \osc(\varphi) D
 = 4C^2n\rho^{n-1} \osc(\varphi) D,
\end{align*}
and consequently,
\begin{align}
\|g_{s}-g_{s'}\|_\infty
&\le \sum_{n=0}^\infty \|P_{s}^{n+1}\bar{\varphi}-P_{s'}^{n+1}\bar{\varphi}\|_\infty \label{eq:poisson-trinagle} \\
&\le 4 C^2 \osc(\varphi) D \sum_{n=0}^\infty (n+1) \rho^{n}  \nonumber\\
&=\frac{4C^2}{(1-\rho)^2}\osc(\varphi) D. \nonumber \qedhere
\end{align}
\end{proof}

\begin{proof}[Proof of Theorem \ref{thm:waning-convergence}]
Thanks to Lemma \ref{lemma:ineq-D}, we have the upper bound:
\begin{equation*}
   \bigg| \frac{A_n}{n^p} \bigg|
   \le \frac{1}{n^p} \sum_{k=1}^n \| g_{S_k} - g_{S_{k-1}} \|_\infty
   \le  \frac{4C^2}{(1-\rho)^2}\osc(\varphi) \bigg( \frac{1}{n^p}\sum_{k=1}^n  D_k \bigg). \qedhere
\end{equation*}
\end{proof}

We are ready to state our main abstract convergence theorem.

\begin{theorem}
Suppose that \ref{a:markov} and \ref{a:simultaneous-uniform} hold and $\varphi\in L^\infty$, and adaptation is $p$-waning strongly/weakly \ref{a:waning}.
\begin{enumerate}[(i)]
   \item \label{item:wlln} If $p=1$, then the weak law of large numbers holds, that is, \eqref{eq:lln} in probability.
   \item \label{item:slln} If $p=1$ strongly, then then the strong law of large numbers holds, that is, \eqref{eq:lln} a.s.
   \item \label{item:clt} If $p=1/2$ and and \ref{a:limiting-poisson} holds with $\sigma_\varphi^2$, then the central limit theorem \eqref{eq:clt} holds with $\sigma_\varphi^2$.
\end{enumerate}
\end{theorem}
\begin{proof}
   For \eqref{item:wlln} and \eqref{item:slln}, consider the decomposition \ref{eq:main-decomposition}, where the terms $M_n/n\to 0$ and $R_n/n\to 0$ a.s.~by Lemma \ref{lem:control-RM}, and $A_n/n\to 0$ in probability/a.s.~due to Theorem \ref{thm:waning-convergence}.

   For the central limit theorem \eqref{item:clt}, we note similarly that $R_n/\sqrt{n}\to 0$ a.s.~by Lemma \ref{lem:control-RM}, $A_n/\sqrt{n}\to 0$ in probability by Theorem \ref{thm:waning-convergence}, and $M_n/\sqrt{n}\to N(0,\sigma_\varphi^2)$ in distribution. The result follows by Slutsky's lemma.
\end{proof}

We then turn into some sufficient conditions which ensure $p$-waning adaptation. The first such condition, \emph{diminishing adaptation} was originally introduced in the coupling context \citep{roberts-rosenthal}. It implies weak 1-waning.

\begin{assumption}[Diminishing adaptation]
   \label{a:diminishing}
   For $D_k$ as in Assumption \ref{a:waning}, 
   $$
      D_k \to 0, \qquad \text{in probability.}
   $$
\end{assumption}

\begin{lemma}
   \label{lem:diminishing-implies-waning}
   If \ref{a:diminishing} holds, then adaptation is weakly 1-waning \ref{a:waning}.
\end{lemma}
\begin{proof}
   Because $D_k\in[0,1]$, \ref{a:diminishing} is clearly equivalent to $\E[D_k] \to 0$. Therefore,
   $$\lim_{n\to\infty} \frac{1}{n} \sum_{k=1}^n \E[D_k] \to 0;
   $$ 
   see Lemma \ref{lem:convergent-serie-for-convergent-sequence} in Appendix \ref{app:calculus}.
\end{proof}

As a consequence, Theorem \ref{thm:waning-convergence} ensures that diminishing adaptation implies the weak law of large numbers (for bounded $\varphi$ when \ref{a:markov} and \ref{a:simultaneous-uniform} hold), which is similar to the original result of \cite{roberts-rosenthal}. If, however, $D_k\to 0$ at a certain rate, strong $p$-waning holds, and therefore strong law of large numbers and/or a central limit theorem can hold.

\begin{assumption}[with rate $p>0$]
   \label{a:diminishing-with-rate}
   For $D_k$ as in Assumption \ref{a:waning}, the following holds:
   $$
      \sum_{k=1}^\infty \frac{\E[D_k]}{k^p} <\infty.
   $$
\end{assumption}
\begin{lemma}
   \label{lem:diminishing-rate-implies-strong-waning}
   If \ref{a:diminishing-with-rate} holds for some $p>0$, then the adaptation is strongly $p$-waning \ref{a:waning}.
\end{lemma}
\begin{proof}
   Let $Z_n = \sum_{k=1}^n D_k/k^p$ and $Z_\infty = \lim_{n\to\infty} Z_n$, where the limit exists in $[0,\infty]$ because $D_k\ge 0$. By monotone convergence, $\E[Z_\infty] = \lim_{n\to\infty} \E[Z_n] < \infty$, so $Z_\infty<\infty$ a.s. We may apply Kronecker's lemma yielding $n^{-p} \sum_{k=1}^n D_k \to 0$ a.s.
\end{proof}

\section{Stochastic approximation type adaptation}
\label{sec:sa}

A natural way to implement adaptations is via stochastic approximation \citep{robbins-monro} or stochastic gradient type optimisation. This type of optimisation was used with Markovian noise earlier \citep[cf.][]{benveniste-metivier-priouret}, the link to adaptive MCMC was established in \citep{andrieu-robert}, and the related theory developed in \citep{andrieu-moulines}; see also the introduction with a number of illustrative examples in \citep{andrieu-thoms}.

Here, in addition to \ref{a:markov}, we assume that $\Sp$ is (subset of) some vector space and the update of $S_k$ has the following recursive form:
\begin{equation}
   S_k = S_{k-1} + \gamma_k H_k,
   \label{eq:sa-gen}
\end{equation}
where $H_k$ is $\F_k$-measurable and $(\gamma_k)_{k\ge 1}$ is a positive step size sequence. 

The rationale of adaptation of the form \eqref{eq:sa-gen} comes from a scenario where the expectation of $H_n$ `under stationarity' is (similar to) the negative gradient of some loss function $L$, which we aim to minimise. In principle, our form of \eqref{eq:sa-gen} can accomodate also $H_n$ which incorporate some sort of averaging/regularisation/speedup for the `stochastic gradients' \citep[e.g.][]{duchi-hazan-singer,kingma-ba}.

\begin{example}[Adaptive Metropolis]
   \label{ex:am}
A recursive form of the adaptive Metropolis algorithm \cite{saksman-am,andrieu-moulines} has the form \eqref{eq:sa-gen}, with
$S_k = (\mu_k, \Sigma_k)$ where $\mu_k$ is the estimated mean vector and $\Sigma_k$ is the covariance matrix, $\gamma_k = k^{-1}$, 
$$
   H_k = \begin{bmatrix} 
      X_k - \mu_{k-1} \\
      X_k X_k^T - \Sigma_{k-1}
   \end{bmatrix},
$$
and $P_{(\mu,\Sigma)} = P_\Sigma$ stand for the Gaussian random-walk Metropolis kernel with proposal covariance $\Sigma$.
\end{example}

\begin{example}[Directional acceptance rate adaptation]
   \label{ex:ram}
In the `robust adaptive Metropolis' algorithm \cite{vihola-ram}, $P_s$ stands also for a random-walk Metropolis kernel with proposal covariance $s$, where proposals are formed as
$$
  Y_k = X_{k-1} + s Z_k, \qquad Z_k \sim N(0,I),
$$
and accepted/rejected with the usual Metropolis acceptence rate $\alpha_k = \min\{1, \frac{\pi(Y_k)}{\pi(X_{k-1})}\}$.
The update uses the random variables $Z_k$:
$$
   H_k = (\alpha_k - \alpha_*)S_{k-1} \frac{Z_k Z_k^T}{\| Z_k \|^2} S_{k-1}^T,
$$
where the desired acceptance rate is typically set to $\alpha_* = 0.234$, and $\gamma_k = n^{-2/3}$ (say).
\end{example}

The usual assumption made in this context is the continuity of the mapping $s\mapsto P_s$, in some suitable sense. We exemplify this with the following assumption, from \cite{andrieu-moulines}:

\begin{assumption}
   \label{a:lipschitz}
There exists a constant $L<\infty$ such that for all $s,s'\in\Sp$:
$$
   \sup_{x} d_\tv\big( P_s(x,\uarg), P_{s'}(x,\uarg)\big) \le L \| s - s' \|.
$$
\end{assumption}
Assumption \ref{a:lipschitz} means that the mappings $s\mapsto P_s(x,\uarg)$ from $\R^d$ to the space of probability measures equipped with the total variation metric, are Lipschitz uniform in $x$.

The following assumption ensures that random-walk Metropolis transitions satisfy \ref{a:lipschitz}.

\begin{assumption}
   \label{a:metropolis}
Suppose $\X\subset\R^d$ is compact, $\pi$ bounded, and $P_s$ is the Gaussian random-walk Metropolis targetting $\pi$ with covariance matrix $s$. Let $0<a<b<\infty$ be constants and let $\Sp\subset\R^{d\times d}$ stand for the symmetric positive definite matrices with all eigenvalues within $[a,b]$.
\end{assumption}

\begin{lemma}
   \label{lem:bounded-rwm-adapt}
   Assumption \ref{a:metropolis} implies \ref{a:simultaneous-uniform} and \ref{a:lipschitz} for  $\{P_s\}_{s\in\Sp}$.
\end{lemma}

For a proof of Lemma \ref{lem:bounded-rwm-adapt}, see for instance \citep[proof of Theorem 1]{saksman-am}.

Assumption \ref{a:metropolis} requires that the feasible covariances $\Sp$ have eigenvalues constrained within $[a,b]$, and both Examples \ref{ex:am} and \ref{ex:ram} can lead, in principle, to either arbitrarily small and/or large eigenvalues for $S_k$. Therefore, we must \emph{constrain} the eigenvalues within some bounds.

\begin{remark}
   \label{rem:stabilisation}
   It may be necessary to add a stabilisation mechanism to the adaptation dynamic \eqref{eq:sa-gen} which ensures that the parameter $S_k$ does not escape the set of feasible values $\Sp$.  A simple way to implement this involves replacing $H_k$ in \eqref{eq:sa-gen} by a modified $\tilde{H}_k$; for instance, we may add an accept/reject type rule:
   \begin{equation}
   \tilde{H}_k = \begin{cases} H_k, & S_{k-1} + \gamma_k H_k \in \Sp, \\
      0, & S_{k-1} + \gamma_k H_k \notin \Sp.
   \end{cases}
   \label{eq:constrained-sa}
   \end{equation}
   This ensures that $S_k$ remains in $\Sp$ and also $\| S_k - S_{k-1} \| \le \gamma \| H_k \|$, which we need below. Alternatively, $S_{k-1} + \gamma_k H_k$ can be projected to $\Sp$ in some way.
\end{remark}

\begin{lemma}
   \label{lem:sa-diminishing}
   Assume that the adaptation has the form \eqref{eq:sa-gen} and \ref{a:simultaneous-uniform} and \ref{a:lipschitz} hold.
   \begin{enumerate}[(i)]
      \item \label{item:sa-diminishing} If $\gamma_k \E\big[\| H_k \|\big] \to 0$, then adaptation is weakly 1-waning \ref{a:waning}.
      \item \label{item:sa-with-rate} If $\sum_{k=1}^\infty \frac{\gamma_k}{k^p} \E\big[\| H_k \|\big] < \infty$ for $p>0$, then adaptation is strongly $p$-waning \ref{a:waning}.
   \end{enumerate}
\end{lemma}
\begin{proof}
Thanks to \ref{a:lipschitz}, we may take $D_k = L \| S_k - S_{k-1} \| = \gamma_k L \| H_k \|$. The condition in \eqref{item:sa-diminishing} implies diminishing adaptation \ref{a:diminishing}, so the claim follows from Lemma \ref{lem:diminishing-implies-waning}, and \eqref{item:sa-with-rate} follows from Lemma \ref{lem:diminishing-rate-implies-strong-waning}.
\end{proof}

\begin{corollary}
   Under \ref{a:metropolis}, the adaptations in Examples \ref{ex:am} and \ref{ex:ram}, using modified $\tilde{H}_k$ as in \eqref{eq:constrained-sa}, satisfy the strong law of large numbers \eqref{eq:lln} for bounded $\varphi$.
\end{corollary}
\begin{proof}
Because $\| H_k \|$ are bounded and $\sum_k \gamma_k/k^{-1}<\infty$ in both Examples \ref{ex:am} and \ref{ex:ram}, the adaptations are strongly 1-waning, so the result follows from Theorem \ref{thm:waning-convergence}.
\end{proof}

We conclude this section by showing that under \ref{a:simultaneous-uniform} and \ref{a:lipschitz}, if $S_n$ converges to a constant $s_\infty\in\Sp$, then a central limit theorem holds.

\begin{lemma}
   \label{lem:convergence-continuity-clt}
   Suppose $\varphi\in L^\infty$, \ref{a:markov} and \ref{a:simultaneous-uniform} hold, and furthermore:
   \begin{enumerate}[(i)]
      \item \label{item:ensure-half-waning} $\sum_{k\ge 1} \frac{\eta_k}{k^{1/2}} \E[\| H_n \|] <\infty$, and
      \item \label{item:convergent-adaptation} $S_k\to s_\infty\in\Sp$ (a.s.),
   \end{enumerate}
   then the central limit theorem \eqref{eq:clt} holds with the asymptotic variance
   \begin{equation}
   \sigma_\varphi^2 = \pi(g_{s_\infty}^2 - (P_{s_\infty} g_{s_\infty})^2) \in [0,\infty).
   \label{eq:sigma-varphi-clt}
   \end{equation}
\end{lemma}
\begin{proof}
   The result follows from Theorem \ref{thm:waning-convergence}, because \eqref{item:ensure-half-waning} implies that adaptation is strongly $1/2$-waning (Lemma \ref{lem:sa-diminishing}), once we establish \ref{a:limiting-poisson} with the claimed $\sigma_\varphi^2$, that is,
   $$
   \frac{1}{n} \sum_{k=1}^n h_{S_{k-1}}(X_{k-1}) \xrightarrow{n\to\infty} \sigma_\varphi^2, \quad\text{where}\quad h_s(x) = P_s g_s^2(x) - (P_s g_s)^2(x).
   $$
   In order to prove this, we will first show that 
   $$
   \| h_s - h_{s'} \|_\infty \le \| P_s g_s^2 - P_{s'} g_{s'}^2\|_\infty + \| (P_s g_s)^2 (P_{s'} g_{s'})^2\|_\infty \le C_h \| s - s' \|.
   $$
   Recall that $g_s$ are bounded (Theorem \ref{thm:bounded}) and Lemma \ref{lemma:ineq-D} ensures that $|g_s(x) - g_{s'}(x)| \le C_1 | s - s'|$ for a constant $C_1<\infty$ that only depends on $\|\varphi\|_\infty$, $C$ and $\rho$ of \ref{a:simultaneous-uniform}.  Therefore,
   $$
   \| P_s g_s^2 - P_{s'} g_{s'}^2\|_\infty 
   \le \| P_s (g_s^2 - g_{s'}^2)\|_\infty + \| (P_s-P_{s'}) g_{s'}^2\|_\infty
   \le \| g_s^2 - g_{s'}^2\|_\infty + L \| s - s'\|  \mathrm{osc}(g_{s'}^2),
   $$
   and because $x\mapsto x^2$ is Lipschitz on compacts, we conclude that $\| P_s g_s^2 - P_{s'} g_{s'}^2\|_\infty \le C_1 \| s - s' \|$.
   Observe that the upper bound \eqref{eq:poisson-trinagle} in Lemma \ref{lemma:ineq-D} holds also for $\| P_s g_s - P_{s'} g_{s'} \|_\infty$, so $|(P_s g_s)^2(x) - (P_{s'} g_{s'})^2(x)| \le C_2 \| s- s'\|$. 

   Let us write
   $$
   \frac{1}{n} \sum_{k=1}^n h_{S_{k-1}}(X_{k-1}) = \frac{1}{n} \sum_{k=1}^n h_{s_\infty}(X_{k-1})
   + \frac{1}{n} \sum_{k=1}^n \big[h_{S_{k-1}}(X_{k-1}) - h_{s_\infty}(X_{k-1})\big].
   $$
   Because $h_{s_\infty}\in L^\infty$, the first term converges to $\pi(h_{s_\infty}) = \sigma_\varphi^2$ by the strong law of large numbers. For the second term, we observe that
   \begin{equation*}
   \bigg| \frac{1}{n} \sum_{k=1}^n \big[h_{S_{k-1}}(X_{k-1}) - h_{s_\infty}(X_{k-1})\big]\bigg|
   \le  \frac{1}{n} \sum_{k=1}^n C_h \| S_{k-1} - s_\infty \|
   \xrightarrow{n\to\infty} 0. \qedhere
   \end{equation*}
\end{proof}

Proving the convergence $S_n\to s_\infty$ as required in Lemma \ref{lem:convergence-continuity-clt} is out of the scope of this paper, but we discuss the matter later in Section \ref{sec:convergence}

\section{Increasingly rare adaptation}
\label{sec:air}

While an continuity condition similar to \ref{a:lipschitz} often holds in practice, it may be difficult to prove (or may not hold at all). For such situations, it may be easier (or necessary) to modify the type of adaptation, to ensure its validity. Recently, it was suggested that adaptations can be applied \emph{increasingly rarely} \citep{chimisov-latuszynski-roberts}.
Here, the adaptation is only applied at pre-specified times, less often as the simulation progresses, and allows to bypass any requirement of continuity. This leads to a general and safe recipe for developing adaptive algorithms for a variety of complex scenarios.

The focus in earlier studies on the topic \citep{chimisov-latuszynski-roberts,hofstadler-latuszynski-roberts-rudolf} has been on constant or independently randomised adaptation times. We consider the following setting, where adaptation times can be random, and may depend on the adaptation process. We assume only \ref{a:markov}, and define two random variables:
\begin{enumerate}
   \item the `change' events $C_k = \{S_k \neq S_{k-1}\}$, and
   \item the `adaptation times' $\tau_1 = \inf \{ k\ge 1\given C_k \}$ and 
   $\tau_j = \inf\{k> \tau_{j-1} \given C_k \}$ for $j\ge 2$,   
\end{enumerate}
with the convention that $\tau_j = \infty$ if $\tau_{j-1}=\infty$ or if no such $k$ exists, and we denote $\tau = \{\tau_j\}_{j\ge 1}$. 

\begin{assumption}
   \label{a:update-times}
    For some $p>0$, one of the following holds:
$$
 \sum_{k=1}^\infty \frac{\P(C_k)}{k^p} < \infty \qquad \text{or}\qquad 
 \sum_{j=1}^\infty \E\big[ \tau_j^{-p} \big] < \infty.
$$
\end{assumption}

\begin{lemma}
   \label{lem:air-waning}
Assumption \ref{a:update-times} implies strong $p$-waning \ref{a:waning}.
\end{lemma}
\begin{proof}
   We may clearly take $D_k = 1(C_k)$ in \ref{a:waning}, and write
   $$
      \sum_{k=1}^\infty \frac{\E[D_k]}{k^p} = \sum_{k=1}^\infty \frac{\P[C_k]}{k^p} < \infty,
   $$
   if the first condition of \ref{a:update-times} holds. For the latter, write
   $$
         \P[C_k] = \P(k\in \tau) = \sum_{j=1}^\infty \P(k = \tau_j),
   $$
   so, changing the order of summation by monotone convergence, we have
   $$
   \sum_{k=1}^\infty \frac{\P[C_k]}{k^p}
   = \sum_{j=1}^\infty \sum_{k=1}^\infty \frac{\P(k = \tau_j) }{k^p}  
   = \sum_{j=1}^\infty \E\bigg[ \frac{1}{\tau_j^p}\bigg] <\infty,
   $$
   Lemma \ref{lem:diminishing-rate-implies-strong-waning} ensures $p$-strong waning.
\end{proof}

Lemma \ref{lem:air-waning} with Theorem \ref{thm:waning-convergence} implies the following:

\begin{theorem}
Suppose that \ref{a:markov}, \ref{a:simultaneous-uniform} and \ref{a:update-times} hold for $p=1$, 
then, for any bounded $\varphi$, the strong law of large numbers \eqref{eq:lln} holds.
\end{theorem}

The following example illustrates that even very slowly increasing inter-adaptation times $\tau_k - \tau_{k-1}$ are enough to guarantee a law of large numbers.

\begin{example}
If $\tau_k = \sum_{j=1}^k n_j$ with increments $n_j \ge c \log^{1+\epsilon}(j)$ for some $c,\epsilon>0$, then it is not hard to verify that $\tau_k \ge c' k\log^{1+\epsilon}(k)$ for some $c'>0$, at least for $k$ large enough. Therefore, \ref{a:update-times} holds with $p=1$.
\end{example}

The following type `hybrid' scheme between stochastic approximation and `increasingly rare type' stochastic optimization scheme was used in \citep{andrieu-vihola} for maximum likelihood estimation with particle independent Metropolis--Hastings \citep{andrieu-doucet-holenstein}.

\begin{example}
Suppose that $S_k$ is updated with a scheme similar to \eqref{eq:sa-gen}:
$$
   S_k = S_{k-1} + \Gamma_k H_k,
$$
where $\Gamma_k$ are random step sizes defined as
$$
\Gamma_k = \gamma_k 1(U_k \le \eta_k),
$$
where $U_k$ are independent uniform $[0,1]$ and $\eta_k$ are `activation' probabilities. 
If $\eta_k  \le c \log^{-(1+\epsilon)}(k)$, we may estimate directly
$$
   \sum_{k=1}^\infty \frac{\P[C_k]}{k}  \le \sum_{k=1}^\infty \frac{c}{k \log^{1+\epsilon}(k)} < \infty,
$$
so \ref{a:update-times} holds with $p=1$.
\end{example}

Increasingly rare adaptation is tempting because it avoids many technical aspects in the theory, in particular, completely bypassing continuity requirement such as \ref{a:lipschitz}. In practice, it may result in less efficient \emph{initial} adaptation. Namely, if $S_k$ is at a bad value, which is kept constant for some time, computational effort can be lost. Continuous adaptations can be more efficient in this regard. It may, however, be, that increasingly rare adaptation is \emph{pre-asymptotically} better: once $S_k$ has already settled to a good value, increasingly rare adaptation can have less perturbations due to the adaptation.  For this reason, combinations of initial continuous adaptation, which becomes increasingly rare later (perhaps in an adaptive manner!) can be worth investigating, too. Theorem \ref{thm:waning-convergence} suggests convergent adaptations, continuous or increasingly rarely adapted, are equally efficient \emph{asymptotically}.

\section{Inhomogeneous Markov chains and external adaptation}
\label{sec:inhomog}

So far, we have discussed fairly general adaptive MCMC schemes. While the terms $R_n$ and $M_n$ can often be controlled quite precisely, the term $A_n$ can only be upper bounded, and usually the upper bound is relatively pessimistic. Sampling with conventinal MCMC, that is, using fixed $P_{S_k}=P_s=P$ for all $k\ge 1$ fits our setting. Because $S_k \equiv s$, the adaptation term vanishes: $A_n \equiv 0$. In this case, we may take $D_k \equiv 0$ in \ref{a:waning}, so our estimates are sharp, too.

In contrast, if $S_k$ come from an exogenous process (independent of $X_k$), we essentially reduce to sampling with an inhomogeneous Markov chain with transitions $P_{s_1}, P_{s_2},\ldots$ Of course, if $D_k \to 0$ in a suitable manner, so that waning adaptation \ref{a:waning} holds, our results apply. However, \ref{a:waning} is not necessary for ergodicity. If $S_k=s_k$ are fixed, each term of $A_k$ is zero-mean under stationarity (because $\pi(g_s)=0$ by Theorem \ref{thm:bounded}). Therefore, one expects that $A_n/n\to 0$ can hold even if all $s_k$ differ, and even if all $D_k$ \ref{a:waning} are bounded away from zero: $D_k > \epsilon >0$ for all $k\ge 1$. This is indeed true under a slightly stronger condition than \ref{a:simultaneous-uniform}:
\begin{assumption}[Uniform Dobrushin--Doeblin condition]
   \label{a:dobrushin}
Suppose that there exists $\beta<1$ such that for all $s\in\Sp$
$$
   \sup_{x,y\in\X} d_\tv \big(P_s(x,\uarg), P_s(y,\uarg) \big)
   = \sup_{\mu \neq \nu} \frac{d_\tv \big(\mu P_s, \nu P_s \big)}{d_\tv(\mu,\nu)}
   = \sup_{\osc(f)\le 1} \osc(P_s f) \le \beta.
$$
\end{assumption}
It is not hard to see that \ref{a:dobrushin} implies \ref{a:simultaneous-uniform} with $C=1$ and $\rho=\beta$, but not vice versa.

\begin{proposition}
   \label{prop:natural}
   Suppose that \ref{a:dobrushin} holds and $S_k=s_k$ are constant and $\varphi\in L^\infty$. Then, $A_n/n\to 0$ in probability and therefore a weak law of large numbers \eqref{eq:lln} holds.
\end{proposition}
\begin{proof}
   Because $g_{s}(x)\in L_0^\infty(\pi)$ and $\| g_s\|_\infty \le (1-\beta)^{-1} \mathrm{osc}(\varphi)$ by Theorem \ref{thm:bounded}, the functions $h_k(x) = g_{s_k}(x) - g_{s_{k-1}}(x) \in L_0^\infty(\pi)$ and $\| h_k \|_\infty \le C' = 2 (1-\beta)^{-1} \mathrm{osc}(\varphi)$. We may write $A_n = n^{-1} \sum_{k=1}^n h_k(X_k)$, so 
   $$
   \E[A_n^2] 
   = \frac{1}{n^2} \sum_{k,j=1}^n \E[h_k(X_k)h_j(X_j)] 
   = \frac{1}{n^2} \bigg[ \sum_{k=1}^n \E[h_k^2(X_k)] + 2 \sum_{k=1}^{n-1} \sum_{j=k+1}^n  \E[h_k(X_k)h_j(X_j)]\bigg].
   $$
   The first sum is bounded by $n (C')^2$, and the expectations in the latter sum satisfies
   $$
   \E[h_k(X_k)h_j(X_{j})]
   = \E\big[ h_k(X_k) \E[h_j(X_j)\mid \F_k]\big]
   = \E\big[ h_k(X_k) P_{s_k}\cdots P_{s_{j-1}} h_j(X_k)\big]
   $$
   Let $h_{k:j}(x) = P_{s_k}\cdots P_{s_{j-1}} h_j(x)$, then 
   Assumption \ref{a:dobrushin} yields 
   $$ 
   \osc( h_{k:j} ) 
   \le \beta^{j-k} \osc( h_j ),
   $$
   and because $\pi(h_{k:j}) = \pi(h_j) = 0$, we have $\inf_x h_{k:j}(x)\le 0$ and $\sup_x h_{k:j}(x) \ge 0$, so $\| h_{k:j} \|_\infty \le \osc(h_{k:j})$. Therefore $\big|\E[h(X_k)h(X_j)]\big| \le \| h_k h_{k:j} \|_\infty \le  (C')^2 \beta^{j-k}$, and so
   $$
   \bigg| \sum_{j=k+1}^n  \E[h(X_k)h(X_j)]\bigg|
   \le (C')^2 \sum_{j=k+1}^n  \beta^{j-k} 
   \le (C')^2 \sum_{h=1}^\infty \beta^{h}
   = \frac{(C')^2\beta}{1-\beta}.
   $$
   We conclude that
   \begin{equation*}
   \E[A_n^2] \le \frac{(C')^2}{n} \bigg[ 1 + \frac{2\beta}{1-\beta}\bigg] \xrightarrow{n\to\infty} 0.
   \qedhere
   \end{equation*}
\end{proof}

The weak law of large numbers as in Proposition \ref{prop:natural} is well-known within the analysis of inhomogeneous Markov chains. It requires \ref{a:dobrushin} which is slightly stronger than \ref{a:simultaneous-uniform} with $C=1$. It is natural to ask whether we could establish $A_n/n\to 0$ with just \ref{a:simultaneous-uniform}? The following example shows that the answer is negative in general.
\begin{example}
   \label{ex:cyclic-chains}
   The diagrams in Figure \ref{fig:cyclic-chains} (a) and (b) illustrate two three-state Markov chains $P_a$ and $P_b$, respectively, which both admit $\pi = (1/2,1/4,1/4)$ as their stationary distribution and are irreducible and aperiodic, and therefore uniformly ergodic, which implies that $\{P_a,P_b\}$ satisfies \ref{a:simultaneous-uniform} (but not \ref{a:dobrushin}). Setting $S_k = a$ for odd $k$ and $S_k=b$ for even $k$, and starting at $X_0=2$ results in $X_k=3$ for odd $k$ and $X_k=2$ for even $k$, and therefore $\varphi(x) = 1(x=1)$ fails the law of large numbers \eqref{eq:lln}.
\end{example}

\begin{figure}
   \begin{tabular}{cc}
\begin{tikzpicture}[shorten >=1pt, node distance=2cm, on grid, auto] 
   \node[state] (A1)   {$1$}; 
   \node[state] (A2) [right=of A1] {$2$}; 
   \node[state] (A3) [right=of A2] {$3$}; 
   \path[->] 
    (A1) edge  node {0.5} (A2)
    (A2) edge  node {1.0} (A3)
    (A3) edge[bend left]  node {1.0} (A1);
    \path[->] 
    (A1) edge[loop left]  node {0.5} (A1);
\end{tikzpicture}
&
\begin{tikzpicture}[shorten >=1pt, node distance=2cm, on grid, auto] 
   \node[state] (B1) {$1$}; 
   \node[state] (B2) [right=of B1] {$2$}; 
   \node[state] (B3) [right=of B2] {$3$}; 
    \path[->] 
    (B2) edge  node [above] {1.0} (B1)
    (B3) edge  node [above] {1.0} (B2)
    (B1) edge[bend right] [below] node {0.5} (B3);
    \path[->] 
    (B1) edge[loop left]  node {0.5} (B1);
\end{tikzpicture}
\\ 
(a) & (b)
\end{tabular}
\caption{Example of simultaneously uniformly ergodic transitions whose compositions do not satisfy a law of large numbers.}
\label{fig:cyclic-chains}
\end{figure}
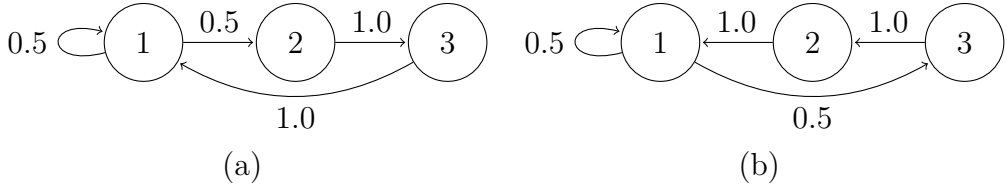

Proposition \ref{prop:natural} illustrates that with constant, or equivalently, exogeneous adaptation, $D_k$ need not necessarily be controlled for a law of large numbers, but as Example \ref{ex:cyclic-chains} illustrates, some structure is still necessary. The zero-mean property of the terms of $A_n$ under independent adaptation suggests that it can introduce less perturbations to the averages, and if adaptations are \emph{nearly} independent, it is plausible that there are less perturbations, too. This can give some theoretical insight to the observed robustness and good empirical performance of parallel chain adaptation \cite[cf.][]{craiu-rosenthal-yang}.

\section{Extensions}
\label{sec:extensions}

The main purpose of this paper is to provide an accessible introduction to the convergence of averages coming from adaptive MCMC schemes. We did not discuss \emph{ergodicity} of adaptive MCMC as defined in \citep{roberts-rosenthal}, that is, convergence of the marginal law of $X_n$ to $\pi$. It is related to the convergence of the averages, but generally a different question: see \citep{chimisov-latuszynski-roberts} for an example of adaptation that satisfies a law of large numbers but does not converge in law.

In addition to different forms of convergence, there are a number of ways the theory presented above can be generalised and extended. We discuss briefly some such extensions below.

\subsection{Beyond simultaneous uniform ergodicity}
\label{sec:non-uniform}

The approach discussed above extends almost directly to the case where the Markov chains satisfy a simultaneous geometric drift and minorisation conditions, as explored in the original paper \citep{andrieu-moulines} which also used the decomposition \eqref{eq:main-decomposition}.  Such conditions involve a `drift' function $V:\X\to[1,\infty)$, and hold quite generally for random-walk chains, in which case $V \simeq \pi^{-\beta}$ for some $\beta\in(0,1)$ \citep[cf.][]{jarner-hansen}. Similar calculations as above carry over to this scenario, when we replace $\| f \|_\infty$ with the weighted supremum norm $\| f \|_V = \sup_x |f(x)|/V(x)$ and the total variation with its analougously defined $V$-weighted variant. 

The recent work of \cite{hofstadler-latuszynski-roberts-rudolf}, which focuses on increasingly rare adaptation, considers $\{P_s\}_{s\in\Sp}$ that are uniform Wasserstein-type contractions; see also the very recent work \cite{brown-rosenthal} which introduces a related `weak' diminishing adaptation. This is a natural generalisation of the presented approach, because total variation is the Wasserstein-1 distance in the case of the discrete metric.

Convergence of ergodic averages can also be verified with weaker than geometric ergodicity/strict contraction assumptions. For instance, a polynomial drift and minorisation can ensure convergence \citep{atchade-fort}. The analysis in sub-geometric scenarios is typically more technical, and the approach in \cite{atchade-fort} is based on \emph{resolvents} $g_s^\epsilon = \sum_{k=0}^\infty (1-\epsilon)^k P^k \bar{\varphi}$ for $\epsilon\in(0,1)$, which exist more generally than solutions of Poisson's equation ($\epsilon=0$).

\subsection{Stability}

In adaptive MCMC, some control of ergodicity is necessary to ensure validity. Our condition \ref{a:simultaneous-uniform} simply requires all Markov transitions to be  `equally well mixing'. Roberts and Rosenthal \citep{roberts-rosenthal} introduced the following quantity, defined for any $\epsilon>0$, to quantify mixing:
$$
  M_\epsilon(x,s) = \inf \{ n\ge 1\given  d_\tv\big( P_s^n(x,\uarg), \pi\big) \le \epsilon \},
$$
and the so-called \emph{containment} condition: $(M_\epsilon(X_n,S_n))_{n\ge 1}$ is bounded in probability, which together with \ref{a:diminishing} ensures (at least) ergodicity \cite[Theorem 2]{roberts-rosenthal}. Proving containment, or some other form of stability, is known to be a notoriously difficult task. A simpler way to ensure this is to investigate $\tilde{M}_\epsilon(x) = \sup_{s\in\Sp} M_\epsilon(x,s)$ instead, but for this to behave well, we typically have to restrict $\Sp$, like in the case of random-walk Metropolis chains \ref{a:metropolis}.

There are some extensions, which avoid such restrictions on the mixing properties of $\{P_s\}_{s\in\Sp}$. The reprojection approach introduced in \cite{andrieu-moulines} involves a growing sequence of sets $\Sp_1\subset\Sp_2\subset\cdots$, each in which a uniform geometric drift and minorisation hold. We start at $\Sp_1$, and if $S_k$ `tries to escape' $\Sp_1$, we `restart' $S_k=s_0\in\Sp_1$ and move to `active' set $\Sp_2$, etc. This is a general scheme, but does come with a number of technical assumptions about the adaptation dynamic. A related approach in \citep{andrieu-vihola} ensures that $S_k\in\Sp_k$, where $\Sp_k$ are growing slow enough, so that loss of ergodicity is controlled.

For some adaptive schemes, it is possible to ensure stability without such added stabilisation mechanisms. For instance, \cite{saksman-vihola} show that the adaptive Metropolis of \ref{ex:am} can be valid on $\X=\R^d$ without upper bounding the covariance, and \cite{vihola-asm} show that acceptance rate adaptations can be stable, too. In some cases, it is possible to develop a joint drift function for $(S_k,X_k)$, which ensures stability \cite{andrieu-tadic-vihola}.

\subsection{Convergence of adaptation}
\label{sec:convergence}

The main focus of this paper has been to ensure a strong law of large numbers with adaptive MCMC. Because many adaptive MCMC are motivated by optimisation (of some proxy that guarantees good mixing), it is natural to expect that many adaptive schemes indeed converge, that is, $S_n\to s_\infty\in\Sp$ (a.s.). Convergence is also a requirement for establishing a central limit theorem, and the limiting value $s_\infty$ dictates the asymptotic efficiency (Lemma \ref{lem:convergence-continuity-clt}).

Stochastic approximation type algorithms (Section \ref{sec:sa}) are the the most natural setting for convergence and indeed \cite{sa-verifiable} establish convergence within the reprojection approach of \cite{andrieu-moulines}. Convergence can also be verified with the expanding projections approach of \citep{andrieu-vihola}, or the joint drift setting of \citep{andrieu-tadic-vihola}, under suitable conditions.

In simple settings, where we enforce $S_k\in\Sp$ by modifying $H_n$ in some way (cf. Remark \ref{rem:stabilisation}), the modifications can be complicated to control. There are some results in the literature on the convergence of projected stochastic approximation, which have issues (personal communication with Błażej Miasojedow and Michał Borowski).

\section{Discussion}
\label{sec:discussion}

We discussed the convergence of averages of adaptive MCMC chains under a general framework \ref{a:markov}, under simultaneous uniform ergodicity \ref{a:simultaneous-uniform} and waning adaptation \ref{a:waning}, which unifies the analysis of stochastic approximation type algorithms (Section \ref{sec:sa}) and other, possibly non-recursively defined schemes, which might be increasingly rarely adapted (Section \ref{sec:air}). We hope that we have convinced the reader about the usefulness of the martingale decomposition \eqref{eq:main-decomposition} in the analysis of various types of adaptive MCMC. In particular, because the `martingale' terms $(M_n)$ and the telescoping `residual' terms $(R_n)$ are usually easy to analyse, and the adaptation terms $(A_n)$ capture the potential perturbations to the averages coming from the adaptation.

There are some adaptive MCMC algorithms which do not fit our framework, beyond those extensions discussed in Section \ref{sec:extensions}.
In some adaptive algorithms, the kernels have different invariant distributions $\pi_s$. This happens often with auxiliary variable schemes, such as with tempering \citep{miasojedow-moulines-vihola} and chains that are allowed to interacting with their past \citep{fort-moulines-priouret}. In such a case, a similar decomposition to \eqref{eq:main-decomposition}, but with additional terms, can still be useful. The same type of decomposition has been extensively used with more general stochastic approximation, which are driven by Markovian noise \citep[cf.][]{benveniste-metivier-priouret}.

There are also many MCMC algorithms which involve adaptation, but are not adaptive MCMC in the sense of this paper. For instance, it is possible to use population of MCMC samplers that use information of other samplers in the proposal \cite[e.g][and references therein]{griffin-walker,sprungk-weissmann-zech}. Some of such samplers may be shown to define Markov chains with the right (marginal) invariant distributions, so there is no need for convergence theory as with adaptive MCMC. However, many of these \emph{algorithms} might benefit from a suitable dependence on the chain's past. We are looking forward to see more `hybrid' algorithms, which can potentially enjoy full efficiency of learning as in adaptive MCMC, but also with population decreasing the perturbation of the adaptation to averages (see also the discussion in Section \ref{sec:inhomog}).

\section*{Acknowledgements}

This work was supported by Research Council of Finland (Finnish Centre of Excellence in Randomness and Structures, grants 346311 and 364216). 


\begin{thebibliography}{43}
   \providecommand{\natexlab}[1]{#1}
   \providecommand{\url}[1]{\texttt{#1}}
   \expandafter\ifx\csname urlstyle\endcsname\relax
     \providecommand{\doi}[1]{doi: #1}\else
     \providecommand{\doi}{doi: \begingroup \urlstyle{rm}\Url}\fi
   
   \bibitem[Andrieu and Moulines(2006)]{andrieu-moulines}
   C.~Andrieu and {\'E}.~Moulines.
   \newblock On the ergodicity properties of some adaptive {MCMC} algorithms.
   \newblock \emph{Ann. Appl. Probab.}, 16\penalty0 (3):\penalty0 1462--1505,
     2006.
   
   \bibitem[Andrieu and Robert(2001)]{andrieu-robert}
   C.~Andrieu and C.~P. Robert.
   \newblock Controlled {MCMC} for optimal sampling.
   \newblock Technical Report Ceremade 0125, Universit{\'e} Paris Dauphine, 2001.
   
   \bibitem[Andrieu and Thoms(2008)]{andrieu-thoms}
   C.~Andrieu and J.~Thoms.
   \newblock A tutorial on adaptive {MCMC}.
   \newblock \emph{Statist. Comput.}, 18\penalty0 (4):\penalty0 343--373, 2008.
   
   \bibitem[Andrieu and Vihola(2014)]{andrieu-vihola}
   C.~Andrieu and M.~Vihola.
   \newblock Markovian stochastic approximation with expanding projections.
   \newblock \emph{Bernoulli}, 20\penalty0 (2):\penalty0 545--585, 2014.
   
   \bibitem[Andrieu et~al.(2005)Andrieu, Moulines, and Priouret]{sa-verifiable}
   C.~Andrieu, {\'E}.~Moulines, and P.~Priouret.
   \newblock Stability of stochastic approximation under verifiable conditions.
   \newblock \emph{{SIAM} J. Control Optim.}, 44\penalty0 (1):\penalty0 283--312,
     2005.
   
   \bibitem[Andrieu et~al.(2010)Andrieu, Doucet, and
     Holenstein]{andrieu-doucet-holenstein}
   C.~Andrieu, A.~Doucet, and R.~Holenstein.
   \newblock Particle {M}arkov chain {M}onte {C}arlo methods.
   \newblock \emph{J. R. Stat. Soc. Ser. B Stat. Methodol.}, 72\penalty0
     (3):\penalty0 269--342, 2010.
   
   \bibitem[Andrieu et~al.(2015)Andrieu, Tadi{\'c}, and
     Vihola]{andrieu-tadic-vihola}
   C.~Andrieu, V.~B. Tadi{\'c}, and M.~Vihola.
   \newblock On the stability of some controlled {M}arkov chains and its
     applications to stochastic approximation with {M}arkovian dynamic.
   \newblock \emph{Ann. Appl. Probab.}, 25\penalty0 (1):\penalty0 1--45, 2015.
   
   \bibitem[Andrieu et~al.(2018)Andrieu, Lee, and Vihola]{andrieu-lee-vihola}
   C.~Andrieu, A.~Lee, and M.~Vihola.
   \newblock Uniform ergodicity of the iterated conditional {SMC} and geometric
     ergodicity of particle {G}ibbs samplers.
   \newblock \emph{Bernoulli}, 24\penalty0 (2):\penalty0 842--872, 2018.
   
   \bibitem[Andrieu et~al.(to appear)Andrieu, Lee, Power, and
     Wang]{andrieu-lee-power-wang}
   C.~Andrieu, A.~Lee, S.~Power, and A.~Wang.
   \newblock Explicit convergence bounds for {M}etropolis {M}arkov chains:
     isoperimetry, spectral gaps and profiles.
   \newblock \emph{Ann. Appl. Probab.}, to appear.
   
   \bibitem[Atchad{\'e} and Rosenthal(2005)]{atchade-rosenthal}
   Y.~F. Atchad{\'e} and J.~S. Rosenthal.
   \newblock On adaptive {M}arkov chain {M}onte {C}arlo algorithms.
   \newblock \emph{Bernoulli}, 11\penalty0 (5):\penalty0 815--828, 2005.
   
   \bibitem[Atchadé and Fort(2010)]{atchade-fort}
   Y.~Atchadé and G.~Fort.
   \newblock Limit theorems for some adaptive {MCMC} algorithms with subgeometric
     kernels.
   \newblock \emph{Bernoulli}, 16\penalty0 (1):\penalty0 116--154, 2010.
   
   \bibitem[Benveniste et~al.(2012)Benveniste, M{\'e}tivier, and
     Priouret]{benveniste-metivier-priouret}
   A.~Benveniste, M.~M{\'e}tivier, and P.~Priouret.
   \newblock \emph{Adaptive algorithms and stochastic approximations}, volume~22.
   \newblock Springer Science \& Business Media, 2012.
   
   \bibitem[Billingsley(1979)]{Billingsley}
   P.~Billingsley.
   \newblock \emph{Probability and Measure}.
   \newblock John Wiley \& Sons, New York, 1979.
   
   \bibitem[Brown and Rosenthal(2024)]{brown-rosenthal}
   A.~Brown and J.~S. Rosenthal.
   \newblock Weak convergence of adaptive {M}arkov chain {M}onte {C}arlo.
   \newblock Preprint arXiv:2406.00820, 2024.
   
   \bibitem[Chimisov et~al.(2018)Chimisov, Łatuszyński, and
     Roberts]{chimisov-latuszynski-roberts}
   C.~Chimisov, K.~Łatuszyński, and G.~Roberts.
   \newblock Air {M}arkov chain {M}onte {C}arlo.
   \newblock Preprint arXiv:1801.09309, 2018.
   
   \bibitem[Craiu et~al.(2009)Craiu, Rosenthal, and Yang]{craiu-rosenthal-yang}
   R.~V. Craiu, J.~Rosenthal, and C.~Yang.
   \newblock Learn from thy neighbor: Parallel-chain and regional adaptive {MCMC}.
   \newblock \emph{J. Amer. Statist. Assoc.}, 104\penalty0 (488):\penalty0
     1454--1466, 2009.
   
   \bibitem[Duchi et~al.(2011)Duchi, Hazan, and Singer]{duchi-hazan-singer}
   J.~Duchi, E.~Hazan, and Y.~Singer.
   \newblock Adaptive subgradient methods for online learning and stochastic
     optimization.
   \newblock \emph{J. Mach. Learn. Res.}, 12\penalty0 (7), 2011.
   
   \bibitem[Fort et~al.(2011)Fort, Moulines, and Priouret]{fort-moulines-priouret}
   G.~Fort, E.~Moulines, and P.~Priouret.
   \newblock Convergence of adaptive and interacting {M}arkov chain {M}onte
     {C}arlo algorithms.
   \newblock \emph{Ann. Statist.}, 39\penalty0 (6):\penalty0 3262--3289, 2011.
   
   \bibitem[Griffin and Walker(2013)]{griffin-walker}
   J.~E. Griffin and S.~G. Walker.
   \newblock On adaptive metropolis--hastings methods.
   \newblock \emph{Statist. Comput.}, 23\penalty0 (1):\penalty0 123--134, 2013.
   
   \bibitem[Griffin et~al.(2020)Griffin, Łatuszyński, and
     Steel]{griffin-latuszynski-steel}
   J.~E. Griffin, K.~G. Łatuszyński, and M.~F.~J. Steel.
   \newblock In search of lost mixing time: adaptive {M}arkov chain {M}onte
     {C}arlo schemes for {B}ayesian variable selection with very large p.
   \newblock \emph{Biometrika}, 108\penalty0 (1):\penalty0 53--69, 10 2020.
   \newblock ISSN 0006-3444.
   
   \bibitem[Haario et~al.(2001)Haario, Saksman, and Tamminen]{saksman-am}
   H.~Haario, E.~Saksman, and J.~Tamminen.
   \newblock An adaptive {M}etropolis algorithm.
   \newblock \emph{Bernoulli}, 7\penalty0 (2):\penalty0 223--242, 2001.
   
   \bibitem[Haario et~al.(2004)Haario, Laine, Lehtinen, Saksman, and
     Tamminen]{saksman-jrstatsoc}
   H.~Haario, M.~Laine, M.~Lehtinen, E.~Saksman, and J.~Tamminen.
   \newblock {M}arkov chain {M}onte {C}arlo methods for high dimensional inversion
     in remote sensing.
   \newblock \emph{J. R. Stat. Soc. Ser. B Stat. Methodol.}, 66\penalty0
     (3):\penalty0 591--607, 2004.
   
   \bibitem[Haario et~al.(2006)Haario, Laine, Mira, and
     Saksman]{haario-laine-mira-saksman}
   H.~Haario, M.~Laine, A.~Mira, and E.~Saksman.
   \newblock {DRAM}: efficient adaptive {MCMC}.
   \newblock \emph{Statist. Comput.}, 16:\penalty0 339--354, 2006.
   
   \bibitem[Hall and Heyde(1980)]{hall-heyde}
   P.~Hall and C.~C. Heyde.
   \newblock \emph{Martingale Limit Theory and Its Application}.
   \newblock Academic Press, New York, 1980.
   \newblock ISBN 0-12-319350-8.
   
   \bibitem[Hofstadler et~al.(2024)Hofstadler, Łatuszynski, Roberts, and
     Rudolf]{hofstadler-latuszynski-roberts-rudolf}
   J.~Hofstadler, K.~Łatuszynski, G.~O. Roberts, and D.~Rudolf.
   \newblock Almost sure convergence rates of adaptive increasingly rare {M}arkov
     chain {M}onte {C}arlo.
   \newblock Preprint arXiv:2402.12122, 2024.
   
   \bibitem[Jarner and Hansen(2000)]{jarner-hansen}
   S.~F. Jarner and E.~Hansen.
   \newblock Geometric ergodicity of {M}etropolis algorithms.
   \newblock \emph{Stochastic Process. Appl.}, 85\penalty0 (2):\penalty0 341--361,
     2000.
   
   \bibitem[Karppinen and Vihola(2021)]{karppinen-vihola}
   S.~Karppinen and M.~Vihola.
   \newblock Conditional particle filters with diffuse initial distributions.
   \newblock \emph{Statist. Comput.}, 31\penalty0 (3):\penalty0 article 24, 2021.
   
   \bibitem[Kingma and Ba(2015)]{kingma-ba}
   D.~Kingma and J.~Ba.
   \newblock Adam: A method for stochastic optimization.
   \newblock In \emph{International Conference on Learning Representations
     (ICLR)}, San Diega, CA, USA, 12 2015.
   
   \bibitem[Marshall and Roberts(2012)]{marshall-roberts}
   T.~Marshall and G.~Roberts.
   \newblock An adaptive approach to {L}angevin {MCMC}.
   \newblock \emph{Statist. Comput.}, 22:\penalty0 1041--1057, 2012.
   
   \bibitem[Miasojedow et~al.(2013)Miasojedow, Moulines, and
     Vihola]{miasojedow-moulines-vihola}
   B.~Miasojedow, E.~Moulines, and M.~Vihola.
   \newblock An adaptive parallel tempering algorithm.
   \newblock \emph{J. Comput. Graph. Statist.}, 22\penalty0 (3):\penalty0
     649--664, 2013.
   
   \bibitem[Pompe et~al.(2020)Pompe, Holmes, and
     Łatuszyński]{pompe-holmes-latuszynski}
   E.~Pompe, C.~Holmes, and K.~Łatuszyński.
   \newblock A framework for adaptive {MCMC} targeting multimodal distributions.
   \newblock \emph{Ann. Statist.}, 48\penalty0 (5):\penalty0 2930 -- 2952, 2020.
   
   \bibitem[Robbins and Monro(1951)]{robbins-monro}
   H.~Robbins and S.~Monro.
   \newblock A stochastic approximation method.
   \newblock 22:\penalty0 400--407, 1951.
   
   \bibitem[Roberts and Rosenthal(2007)]{roberts-rosenthal}
   G.~O. Roberts and J.~S. Rosenthal.
   \newblock Coupling and ergodicity of adaptive {M}arkov chain {M}onte {C}arlo
     algorithms.
   \newblock \emph{J. Appl. Probab.}, 44\penalty0 (2):\penalty0 458--475, 2007.
   
   \bibitem[Roberts et~al.(1997)Roberts, Gelman, and
     Gilks]{roberts-gelman-gilks-scaling}
   G.~O. Roberts, A.~Gelman, and W.~R. Gilks.
   \newblock Weak convergence and optimal scaling of random walk {M}etropolis
     algorithms.
   \newblock \emph{Ann. Appl. Probab.}, 7\penalty0 (1):\penalty0 110--120, 1997.
   
   \bibitem[Saksman and Vihola(2010)]{saksman-vihola}
   E.~Saksman and M.~Vihola.
   \newblock On the ergodicity of the adaptive {M}etropolis algorithm on unbounded
     domains.
   \newblock \emph{Ann. Appl. Probab.}, 20\penalty0 (6):\penalty0 2178--2203,
     2010.
   
   \bibitem[Shiryaev(1996)]{shiryaev}
   A.~N. Shiryaev.
   \newblock \emph{Probability}.
   \newblock Springer-Verlag, New York, second edition, 1996.
   \newblock ISBN 0-387-94549-0.
   
   \bibitem[Sprungk et~al.(2023)Sprungk, Weissmann, and
     Zech]{sprungk-weissmann-zech}
   B.~Sprungk, S.~Weissmann, and J.~Zech.
   \newblock Metropolis-adjusted interacting particle sampling.
   \newblock Preprint arXiv:2312.13889, 2023.
   
   \bibitem[Tian et~al.(2024)Tian, Lee, and Zhou]{tian-lee-zhou}
   Z.~Tian, A.~Lee, and S.~Zhou.
   \newblock Adaptive tempered reversible jump algorithm for {B}ayesian curve
     fitting.
   \newblock \emph{Inverse Probl.}, 40\penalty0 (4), 2024.
   
   \bibitem[Vihola(2011)]{vihola-asm}
   M.~Vihola.
   \newblock On the stability and ergodicity of adaptive scaling {M}etropolis
     algorithms.
   \newblock \emph{Stochastic Process. Appl.}, 121\penalty0 (12):\penalty0
     2839--2860, 2011.
   
   \bibitem[Vihola(2012)]{vihola-ram}
   M.~Vihola.
   \newblock Robust adaptive {M}etropolis algorithm with coerced acceptance rate.
   \newblock \emph{Statist. Comput.}, 22\penalty0 (5):\penalty0 997--1008, 2012.
   
   \bibitem[Vihola and Franks(2020)]{vihola-franks-abc}
   M.~Vihola and J.~Franks.
   \newblock On the use of approximate {B}ayesian computation {M}arkov chain
     {M}onte {C}arlo with inflated tolerance and post-correction.
   \newblock \emph{Biometrika}, 107\penalty0 (2):\penalty0 381--395, 2020.
   
   \bibitem[Wang et~al.(2024)Wang, Chen, Kanagawa, and
     Oates]{wang-chen-kanagawa-oates}
   C.~Wang, W.~Chen, H.~Kanagawa, and C.~J. Oates.
   \newblock Reinforcement learning for adaptive {MCMC}.
   \newblock Preprint arXiv:2405.13574, 2024.
   
   \bibitem[Łatuszyński et~al.(2013)Łatuszyński, Roberts, and
     Rosenthal]{latuszynski-roberts-rosenthal}
   K.~Łatuszyński, G.~O. Roberts, and J.~S. Rosenthal.
   \newblock Adaptive {G}ibbs samplers and related {MCMC} methods.
   \newblock \emph{Ann. Appl. Probab.}, 23\penalty0 (1):\penalty0 66--98, 2013.
   
   \end{thebibliography}

\appendix

\section{Details about measurability}
\label{app:measurability}

Proof of the following theorem can be found for instance in \citep[Theorem 3.4]{Billingsley}.
\begin{theorem}[Monotone Class]\label{theorem:monotone-class-theorem}
If $\mathcal{A}$ is an algebra and $\mathcal{M}$ is a monotone class, then $\mathcal{A}\subset \mathcal{M}$ implies $\sigma(\mathcal{A})\subset \mathcal{M}$.
\end{theorem}

If $(X,\mathcal{X})$ and $(Y,\mathcal{Y})$ are measurable spaces let us denote by $\mathcal{A}_\mathcal{\mathcal{X},\mathcal{Y}}$ the collection of finite disjoint unions of measurable rectangles $A\times B$, where $A\in \mathcal{X}$ and $B\in \mathcal{Y}$. 
The collection $\mathcal{A}_\mathcal{\mathcal{X},\mathcal{Y}}$ is an algebra, which is easily verified.

\begin{corollary}\label{corollary: monotone class theorem}
Let $(X,\mathcal{X})$ and $(Y,\mathcal{Y})$ be measurable spaces. If $\mathcal{M}$ is a monotone class, then $\mathcal{A}_\mathcal{\mathcal{X},\mathcal{Y}}\subset \mathcal{M}\subset \mathcal{X}\otimes \mathcal{Y}$ implies $\mathcal{M}= \mathcal{X}\otimes \mathcal{Y}$.
\end{corollary}

\begin{lemma}\label{lem:(s,x)-measurability}
Let $\{P_s\}_{s\in\Sp}$ be a family of Markov transition probabilities on $\X$ for which the regularity as noted in Remark \ref{rem:measurability} holds.
Then if $h:\Sp\times\X\to \R$ is a bounded $\sigmaSp\otimes\sigmaX$-measurable function, we have that
$$
(s,x)\mapsto \int P_s(x,\ud y)h(s,y)
$$
is bounded $\sigmaSp\otimes\sigmaX$-measurable function.
\end{lemma}
\begin{proof}
Taking the function as indicator function let us denote the collection of $\sigmaSp\otimes\sigmaX$-measurable sets for which the claim holds:
$$
\mathcal{G}=\biggbraces{E\in \sigmaSp\otimes\sigmaX \given 
(s,x)\mapsto \int P_s(x,\ud y)\charfun{(s,y)\in E} \text{ is $\sigmaSp\otimes\sigmaX$-measurable}}.
$$
Let $A\in \Sp$ and $B\in \X$. Then, $(s,x)\mapsto P_s(x,B)$ and $(s,x)\mapsto 
\charfun{s\in A}$ are $\sigmaSp\otimes\sigmaX$-measurable and therefore
\begin{align*}
(s,x)&\mapsto P_s(x,B)\charfun{s\in A} =\int P_s(x,\ud y)\charfun{(s,y)\in A \times B}
\end{align*} 
is $\sigmaSp\otimes\sigmaX$-measurable and by linearity $A_{\sigmaSp,\sigmaX}\subset \mathcal{G}$.
Let $A_1\subset A_2,\subset \dots\in \mathcal{G}$ be an increasing sequence of sets and $B_1\supset B_2\supset \dots \in \mathcal{G}$ a decreasing sequence of sets.
For any $(s,x)\in \Sp\times \X$ and $E\in \sigmaSp\otimes\sigmaX$ the section $E^s=\braces{y\in \X:(s,y)\in E}$ is $\mathcal{X}$-measurable and thus by the monotone convergence theorem we have for the increasing $(A_i)_{i=1}^\infty$, that
\begin{align}\label{equation:measurable-increasing-limits-1}
\begin{split}
\lim_{n\to \infty}\int P_s(x,\ud y)\charfun{(s,y)\in A_n}
&=\int P_s(x,\ud y)\lim_{n\to \infty}\charfun{(s,y)\in A_n} \\
&=\int P_s(x,\ud y)\charfun{(s,y)\in \bigcup_{n=1}^\infty A_n}.
\end{split}
\end{align}
By applying \eqref{equation:measurable-increasing-limits-1} for the complements of $(B_i)_{i=1}^\infty$, it follows that
\begin{align}\label{equation:measurable-decreasing-limits-1}
\lim_{n\to \infty}\int P_s(x,\ud y)\charfun{(s,y)\in B_n}=\int P_s(x,\ud y)\charfun{(s,y)\in \bigcap_{n=1}^\infty B_n}.
\end{align}
As a limit of $\sigmaSp\otimes\sigmaX$-measurable functions both \eqref{equation:measurable-increasing-limits-1} and \eqref{equation:measurable-decreasing-limits-1} are $\sigmaSp\otimes\sigmaX$-measurable functions and thus $\mathcal{G}$ is a monotone class. By Corollary \ref{corollary: monotone class theorem} we have $\mathcal{G}=\sigmaSp\otimes\sigmaX$.
Therefore, the claim holds for simple functions, and this generalises to positive functions by simple function approximation from below and the monotone convergence theorem. The generalisation to bounded functions follows as usual, considering the positive and negative parts separately.
\end{proof}

\begin{lemma}
   \label{lem:gs-measurability}
Let $\{P_s\}_{s\in\Sp}$ be a family of Markov transition probabilities on $\X$ for which the regularity as noted in Remark \ref{rem:measurability} holds. Let $f\in L^\infty$ and $k\in \N$, then $(s,x)\mapsto P_s^k f(x)$ is bounded $\sigmaSp\otimes\sigmaX$-measurable function, and $x\mapsto P_s^k f(x)$ is bounded $\sigmaX$-measurable function for all $s\in \Sp$.
\end{lemma}
\begin{proof}
For the first claim, case $k=0$ is clear, since $(s,x)\mapsto P_s^0 f(x)=f(x)$.
Cases $k\ge 1$ follow from Lemma \ref{lem:(s,x)-measurability} by induction, since $P_s^k f(x)=P_sP_s^{k-1}f(x)$. As a section of a $\sigmaSp\otimes\sigmaX$-measurable function, the mapping $x\mapsto P_s^k f(x)$ is $\sigmaX$-measurable.
\end{proof}

\begin{lemma}
\label{lem:h-measurable}
   Let $h:\Sp\times\X\to\R$ be $\sigmaSp\otimes\sigmaX$-measurable and bounded. Then, under \ref{a:markov} and the regularity as noted in Remark \ref{rem:measurability}
   $$
      \E[h(S_k, X_{k+1})\mid \F_k] \aseq \int P_{S_k}(X_k, \ud y) h(S_k, y).
   $$
\end{lemma}
\begin{proof}
Taking the function as indicator function let us denote the collection of $\sigmaSp\otimes\sigmaX$-measurable sets for which the claim holds:
   $$\mathcal{H}=\biggbraces{E\in \sigmaSp\otimes\sigmaX \given \E[\charfun{(S_k,X_{k+1})\in E}\mid \F_k]\aseq \int P_{S_k}(X_k,\ud y)\charfun{(S_k,y)\in E}}.$$
Let $A\in \mathcal{S}$ and $B\in \mathcal{X}$, then by \ref{a:markov} we have almost surely that
   \begin{align*}
   \E[\charfun{(S_k,X_{k+1})\in A\times B}\mid \F_k]
   &= \charfun{S_k\in A}\E[\charfun{X_{k+1}\in B}\mid \F_k] \\
   &= \charfun{S_k\in A}P_{S_k}(X_k,B) \\
   &=\charfun{S_k\in A}\int P_{S_k}(X_k,\ud y)\charfun{y\in B} \\
   &=\int P_{S_k}(X_k,\ud y)\charfun{(S_k,y)\in A\times B}
   \end{align*}
and therefore $A_{\sigmaSp,\sigmaX}\subset \mathcal{H}$ follows from linearity.
Let $A_1\subset A_2,\subset \dots\in \mathcal{H}$ be an increasing sequence of sets and $B_1\supset B_2\supset \dots \in \mathcal{H}$ a decreasing sequence of sets.
We have
   \begin{align}\label{equation:measurable-increasing-limits-2}
   \lim_{n\to \infty} \E\big[\charfun{(S_k,X_{k+1})\in A_n} \mid  \F_k\big]
   \aseq \E\big[\lim_{n\to \infty}\charfun{(S_k,X_{k+1})\in A_n} \mid  \F_k\big]
   \end{align}
   by the monotone convergence theorem and similarly
   \begin{align}\label{equation:measurable-increasing-limits-3}
   \lim_{n\to \infty} \int P_{S_k}(X_k,\ud y) \charfun{(S_k,y)\in A_n}
   =\int P_{S_k}(X_k,\ud y)\lim_{n\to \infty} \charfun{(S_k,y)\in A_n}.
   \end{align}
   Combining \eqref{equation:measurable-increasing-limits-2} and \eqref{equation:measurable-increasing-limits-3} with the definition of $\mathcal{H}$, we obtain
   \begin{align*}
   \E\Bigg[\charfun{(S_k,X_{k+1})\in \bigcup_{n=1}^\infty A_n}\;\Bigg|\;  \F_k\Bigg] 
   \aseq \int P_{S_k}(X_k,\ud y)\charfun{(S_k,y)\in \bigcup_{n=1}^\infty A_n}.
   \end{align*}
By applying the previous equation for the complements of $(B_i)_{i=1}^\infty$, it follows that
   \begin{align*}
   \E\Bigg[\charfun{(S_k,X_{k+1})\in \bigcap_{n=1}^\infty B_n} \;\Bigg|\; \F_k\Bigg]\aseq \int P_{S_k}(X_k,\ud y)\charfun{(S_k,y)\in \bigcap_{n=1}^\infty B_n}.
   \end{align*}
Thus $\mathcal{H}$ is a monotone class and $\mathcal{H}=\sigmaSp\otimes\sigmaX$ by Corollary \ref{corollary: monotone class theorem}. Therefore, the claim holds for simple functions and this can be generalised to any $\sigmaSp\otimes\sigmaX$-measurable bounded function similarly as in the proof of Lemma \ref{lem:(s,x)-measurability}. 
\end{proof}

\section{Results from calculus}
\label{app:calculus}
We state the next simple lemma without proof.
\begin{lemma}\label{lem:convergent-serie-for-convergent-sequence}
Let $(a_i)_{i=1}^\infty$ be a sequence such that $a_i\ge 0$ for all $i$ and $\lim_{i\to \infty}a_i=0$. Then
$$
\frac{1}{n}\sum_{i=1}^n a_i \xrightarrow{n\to \infty}0.
$$
\end{lemma}

Finally, we restate Kronecker's lemma for the reader's convenience.

\begin{lemma}[Kronecker]\label{lemma:Kronecker}
Let $(a_i)_{i=1}^\infty$ be a sequence in $\R$. Let $(c_i)_{i=1}^\infty$ be positive and increasing sequence such that $\lim_{n\to \infty}c_n=\infty$. If
$\sum_{i=1}^\infty \frac{a_i}{c_i}$
is convergent, then we have that
$$
\lim_{n\to \infty}\frac{1}{c_n}\sum_{i=1}^n a_i=0.
$$
\end{lemma}

\section{Results for martingales}
\label{app:martingales}

The following martingale convergence result is well-known, and can be found for instance in \citep[p.~510]{shiryaev}.
\begin{lemma}
   \label{lem:square-martingale-conv}
Let $(M_n)_{n\ge 0}$ be a martingale. If $(M_n)_{n\ge 0}$ is $L^p$ bounded for some $p>1$, that is, $\sup_{n}\E[|M_n|^p]<\infty$, then there exsts an a.s.~finite random number $M_\infty$ such that $M_n\to M_\infty$ a.s.
\end{lemma}

The following martingale central limit is a restatement of \citep[Corollary 3.1]{hall-heyde}.

\begin{theorem}
   \label{thm:martingale-clt}
Let $(M_n)_{n\ge 0}$ be a martingale with respect to $(\F_k)_{k\ge 0}$, such that there exists $\sigma^2\in(0,\infty)$ such that the differences $\Delta_k = M_k - M_{k-1}$ satisfy the following for all $\epsilon>0$: 
\begin{align*}
   \frac{1}{n}\sum_{k=1}^n \E[\Delta_k^2\mid \F_{k-1}] &\xrightarrow{n\to\infty} \sigma^2, \\
   \frac{1}{n} \sum_{k=1}^n \E[\Delta_k^2 1(|\Delta_k| \ge \epsilon \sqrt{n})\mid \F_{k-1}] &\xrightarrow{n\to\infty} 0,
\end{align*}
in probability, then, $n^{-1/2} M_n \to N(0,\sigma^2)$ in distribution as $n\to\infty$.
\end{theorem}

\end{document}